\documentclass[11pt,a4paper,twoside]{amsart}
\usepackage{amssymb,amsmath,amsthm,graphicx,color
}
\theoremstyle{plain}
\newtheorem{theorem}{Theorem}[section]

\newtheorem{lemma}[theorem]{Lemma}

\theoremstyle{remark}
\newtheorem{remark}[theorem]{Remark}

\allowdisplaybreaks[4]

\usepackage{amsrefs}

\numberwithin{equation}{section}

\newcommand{\R}{\mathbb{R}}
\newcommand{\Z}{\mathbb{Z}}

\newcommand{\I}{\infty}

\newcommand{\norm}[1]{\left\lVert #1\right\rVert}
\newcommand{\xLebn}[2]{\left\lVert #1 \right\rVert_{L^{#2}_x}}
\newcommand{\Lebn}[2]{\left\lVert #1 \right\rVert_{L^{#2}}}
\newcommand{\xtLebn}[3]{\left\lVert #1 \right\rVert_{L^{#2}_x L^{#3}_T}}
\newcommand{\txLebn}[3]{\left\lVert #1 \right\rVert_{L^{#2}_T L^{#3}_x}}

\newcommand{\xtLebnT}[3]{\left\lVert #1 \right\rVert_{L^{#2}_x L^{#3}_t}}
\newcommand{\txLebnT}[3]{\left\lVert #1 \right\rVert_{L^{#2}_t L^{#3}_x}}

\newcommand{\Sobn}[2]{\left\lVert #1 \right\rVert_{H^{#2}}}

\newcommand{\Jbr}[1]{\left\langle #1 \right\rangle}

\newcommand{\vertiii}[1]{{\left\vert\kern-0.25ex\left\vert\kern-0.25ex\left\vert #1 
    \right\vert\kern-0.25ex\right\vert\kern-0.25ex\right\vert}}

\def\({\left(}
\def\){\right)}
\def\<{\left\langle}
\def\>{\right\rangle}
\def\le{\leqslant}
\def\ge{\geqslant}

\def\d{{\partial}}

\def \tilde{\widetilde}

\def \l{\lambda}

\def \d{\delta}

\def \pa{\partial}

\def \a{\alpha}
\def \b{\beta}

\def \t{\theta}

\def \P{\Phi}

\def \ga{\gamma}

\newcommand{\p}{\partial}

\DeclareMathOperator{\supp}{supp}

\newcommand{\md}{\color{black}}

\newcommand{\todayd}{\the\year/\the\month/\the\day}

\theoremstyle{definition}

\begin{document}
\title[The generalized KdV type equation]
{On a class of solutions to the generalized KdV type equation}

\author[F. Linares]{Felipe Linares}
\address[]{IMPA, Estrada Dona Castorina 110, Rio de Janeiro 22460-320, RJ Brazil}
\email{linares@impa.br}
\author[H. Miyazaki]{Hayato MIYAZAKI}
\address[]{Advanced Science Course, Department of Integrated Science and Technology, National Institute of Technology, Tsuyama College, Tsuyama, Okayama, 708-8509, Japan}
\email{miyazaki@tsuyama.kosen-ac.jp}
\author[G.Ponce]{Gustavo PONCE}
\address[]{Department of Mathematics, University of California Santa Barbara, Santa Barbara, California 93106, USA}
\email{ponce@math.ucsb.edu}

\date{}

\maketitle
\vskip-5mm
\begin{abstract}
We consider the IVP associated to the generalized  KdV equation with low 
degree of non-linearity
\begin{equation*}
\pa_t u + \pa_x^3 u \pm |u|^{\a}\pa_x u = 0,\; x,t \in \R,\;\a \in (0,1).
\end{equation*}
By using an  argument similar to that introduced by Cazenave and Naumkin \cite{CaNa} we establish the local well-posedness for
a class of  data in an appropriate weighted Sobolev space. Also, we show that the solutions obtained satisfy the propagation of regularity principle proven in  \cite{ILP} in solutions of the $k$-generalized KdV equation.

\end{abstract} 

\section{Introduction}
In this work we study the initial value problem (IVP) for the  generalized Korteweg-de Vries (KdV) type equation
\begin{align}
	\left\{ 
	\begin{aligned}
	&{} \pa_t u + \pa_x^3 u \pm |u|^{\a}\pa_x u = 0,\; x,t \in \R, \;\a \in (0,1),\\
	&{} u(x,0) = u_0(x),
	\end{aligned}
	\right.  \label{gkdv} \tag{GK}
\end{align}
where  $u=u(x,t)$ is a  real-valued or complex-valued unknown function. 

The equation in \eqref{gkdv} is a lower nonlinearity version of the celebrate Korteweg-de Vries equation (KdV) \cite{KdV}
\begin{equation}
\label{01}
\pa_t u + \pa_x^3 u +u\pa_x u = 0,\; x,t \in \R,
\end{equation}
and its $k$-generalized form 
\begin{equation}
\label{02}
\pa_t u + \pa_x^3 u +u^k\pa_x u = 0,\; x,t \in \R,\;\,k\in\Z^+.
\end{equation}
The IVP and the periodic boundary value problem (pbvp) associated to the equation in \eqref{02} have been extensively studied. In fact, sharp local and global 
well-posedness, stability of special solutions and  blow-up results have been established in several publications (for a more detail account of them we refer to \cite{LP} Chapters 7-8).

Formally, real valued solutions of \eqref{gkdv} satisfy three conservation laws:
\begin{align*}
I_1(u)&=\int_{-\infty}^{\infty}u(x,t)dx,\;\;\;I_2(u)=\int_{-\infty}^{\infty}u^2(x,t)dx,\\
I_3(u)&=\int_{-\infty}^{\infty}((\partial_xu)^2\mp \frac{2}{(\alpha+1)(\alpha+2)}|u|^{\alpha+2})(x,t)dx.
\end{align*}

 Roughly speaking,  the nonlinearity in \eqref{gkdv} is non-Lipchitz in any  Sobolev space $H^s(\R) = (1-\pa_x^2)^{s/2}L^2(\R),\,s\in\R,$ or in the weighted versions  $H^s(\R)\cap L^2(\R:\Jbr{x}^rdx)$, $\,k,\,\,r\in\R$ as a consequence local well-posedness can not be established in these spaces.

Our  first goal  is to  establish the local well-posedness for the IVP \eqref{gkdv} in a class of initial data.  To present our result we first describe our motivation and the ideas behind the proofs.

In \cite{CaNa}  Cazenave and Naumkin studied the IVP associated to semi-linear Schr\"odinger equation,
\begin{equation}\label{nls}
\begin{cases}
\p_tu=i(\Delta u\pm|u|^{\a} u), \hskip15pt x\in\R^n,\;t\in\R, \hskip5pt \a>0,\\
u(x,0)=u_0(x),
\end{cases}
\end{equation}
with initial data $u_0\in H^s(\R^n)$.  For every $\a>0$ they constructed a class of initial data for which there exist  unique local solutions for the IVP \eqref{nls}.  Also, they obtained 
a class of initial data for $\alpha>\frac{2}{n}$ for which there exist global solutions that scatter.

One of the ingredients in the proofs
of their results, is the fact that solutions of the linear problem satisfy
\begin{equation}\label{lower}
\underset{x\in\R^n}{\rm Inf}\; \Jbr{x}^m\,| e^{it\Delta} u_0(x)| >0,
\end{equation}
for $t\in[0,T]$ with $T$ sufficiently small whenever the initial data satisfy
\begin{equation}\label{lower-2}
\underset{x\in\R^N}{\rm Inf}\; \Jbr{x}^m\,|u_0(x)|\ge \lambda>0.
\end{equation}
This is reached for $m=m(\a)$ and $u_0\in H^s(\R^n)$ with $s$ sufficiently large with appropriate decay. To prove the inequality \eqref{lower} the authors in \cite{CaNa}  
rely on  Taylor's power expansion to avoid applying the Sobolev embedding
since the nonlinear $|u|^{\alpha} u$ is not regular enough and it would restrict the argument to dimensions $n\ge 4$. 

In \cite{LPS}, the arguments introduced in \cite{CaNa} were modified to study the IVP  associated to the generalized derivative Schr\"odinger equation
\begin{equation*}
\partial_tu=i\partial_x^2u + \mu\,|u|^{\a}\partial_xu, \hskip10pt x,t\in\R,  \hskip5pt 0<\alpha \le 1\;\; {\rm and}\;\; |\mu|=1,
\end{equation*}
establishing local well-posedness for
a class of small data in an appropriate weighted Sobolev space.

In the case considered here, the non-linearity is non-Lipschitz. Motivated by the results in \cite{CaNa} and using the smoothing effects of Kato type \cite{Ka} we shall obtain the desired local well-posedness result
for the IVP \eqref{gkdv} for a class of data satisfying \eqref{lower-2}.

The first aim of this paper is to show the following:
\begin{theorem} \label{thm:1}
Fix $m= \left[ \frac{1}{\a} \right]+1$. Let $s \in \Z^{+}$ satisfy $s \ge 2m+4$. 
Assume $\,u_0\;$ is a complex-valued function such that  
\begin{align}
	&{}\md u_0 \in H^s(\R),\; \Jbr{x}^m u_0 \in L^{\I}(\R),\; \Jbr{x}^m \pa_x^{j+1} u_0 \in L^{2}(\R),\; j=0,1,2,3, \label{thm:14} \\
	&{}\Sobn{u_0}{s} + \Lebn{\Jbr{x}^m u_0}{\I} + {\md \sum_{j=0}^3 \Lebn{\Jbr{x}^m \pa_x^{j+1} u_0}{2}} < \d \label{thm:16}
\end{align}
for some $\d>0$ and
\begin{align}
	\inf_{x \in \R} \Jbr{x}^m |u_0(x)| =: \l >0. \label{thm:15}
\end{align}
Then there exists $T=T(\a;\d;s; \l)>0$ such that \eqref{gkdv} has a unique local solution
\begin{align}
{\md u \in C([0,T]; H^s(\R)),\quad \Jbr{x}^m \pa_x^{j+1} u \in C([0,T]; L^2(\R)),\quad j=0,1,2,3} \label{thm:12}
\end{align}
with
\begin{align}
	\Jbr{x}^m u \in C([0,T]; L^\I(\R)), \quad \pa_x^{s+1} u \in L^{\infty}(\R; L^{2}([0,T])), \label{thm:13}
\end{align}
and
\begin{align}
	\sup_{0 \le t \le T} \Lebn{\Jbr{x}^m (u(t)-u_0)}{\I} \le \frac{\l}2. \label{thm:11}
\end{align}
Moreover, the map $u_0 \mapsto u(t)$ is continuous in the following sense:
For any compact $I \subset [0,T]$, there exists a neighborhood $V$ of $u_0$ satisfying  \eqref{thm:14} and \eqref{thm:15} such that the map 
is Lipschitz continuous from $V$ {\md into the class defined by \eqref{thm:12} and \eqref{thm:13}.} 
\end{theorem}

Above we have used the following notation: $\Jbr{x} = (1+|x|^2)^{1/2},\;x \in \R$,
and for any $x \in \R$, $[x]$ denotes the greatest integer less than or equal to $x$.

\begin{remark} \label{rem:1}
The solution in Theorem \ref{thm:1} is in fact unique in the class
\[
	{\md C([0,T]; H^{2m+2}(\R)),\quad \Jbr{x}^m \pa_x^{j+1} u \in C([0,T]; L^2(\R)),\quad j=0,1,}
\]
with
\[
	\Jbr{x}^m u \in L^{\I}([0,T]; L^\I(\R)), \quad \pa_x^{2m+3} u \in L^{\I}(\R; L^{2}([0,T])),
\]
and
\[
	\sup_{0 \le t \le T} \Lebn{\Jbr{x}^m (u(t)-u_0)}{\I} \le \frac{\l}2.
\]
\end{remark}

\vskip.1in  
\begin{remark}
Inequality \eqref{thm:11} gives us 
\begin{equation}
\label{BB0}
	\frac{\l}{2}\leq  -\frac{\l}{2} +\Jbr{x}^m |u_0(x)|  \le \Jbr{x}^m |u(x,t)| \le \Jbr{x}^m |u_0(x)| + \frac{\l}{2}
	\end{equation}
for any $(x,t) \in \R  \times [0,T]$.

\end{remark}

\vskip.1in  
\begin{remark} As in \cite{CaNa} from  \eqref{thm:15} and \eqref{BB0} it follows that
\begin{equation*}
\Jbr{x}^{m-1/2} u_0\notin L^2(\R)\;\;
\;\;\;\text{and}\;\;\;\;\;\;\Jbr{x}^{m-1/2} u(t)\notin L^2(\R),\;\;\;\;\;t\in(0,T].
\end{equation*}  
\end{remark}

A typical data $u_0$ satisfying the hypotheses in Theorem \ref{thm:1} is:
\begin{equation*}
u_0(x)=\frac{2\lambda \,e^{i\theta}}{\Jbr{x}^m} + \varphi(x),\;\;\;\;\;\;\;\;\;\;\theta\in\R,
\end{equation*}
with $\,\varphi\in \mathcal{S}(\R)$ and 
\begin{equation*}
 \| \Jbr{x}^m \varphi\|_{\infty} \leq \lambda.
\end{equation*}

\vskip.1in  
\begin{remark} We observe that the IVP \eqref{gkdv} has (real) traveling wave solutions (positive, even and radially decreasing) with speed $c>0$
\begin{equation*}
\begin{aligned}
\phi_{c,\alpha}(x,t)&=\,c^{1/\alpha}\,\phi(\sqrt{c}(x-ct)),\\
\\
\phi(x)&=\Big( \frac{(\alpha+1)(\alpha+2)}{2}\,\text{sech}\Big(\frac{\alpha x}{2}\Big)\Big)^{2/\alpha}.
\end{aligned}
\end{equation*}
We observe that the data $\phi$ does not satisfy \eqref{thm:15} so the traveling wave is not in the class of solutions provided by Theorem \ref{thm:1}. 
This is similar to the situation for the log-KdV equation
\begin{equation}
\label{logkdv}
 \pa_t v + \pa_x^3 v +\pa_x( v\,\log|v|)= 0,\; x,t \in \R
\end{equation}
described in \cite{CaPe}, where the well-posedness obtained there does not include the uniqueness and continuous dependence for the gaussian traveling wave solution of \eqref{logkdv}. 

\end{remark}

Our second result is concerned with the propagation of regularity in the right hand side of the data for positive times of real solutions in the class provided by Theorem \ref{thm:1}. 
It affirms that this regularity moves with infinite speed to its left as time evolves. 

\begin{theorem} \label{thm:2}
 In addition to hypotheses \eqref{thm:14}-\eqref{thm:15}, suppose $u_0$ is real valued and there exist $\,l\in \Z^+$ and $\,x_0\in\R$ such that
\begin{align*}
	u_0\Big|_{(x_0,\infty)}\in H^{s+l}((x_0,\infty)). 
\end{align*}
Then for any $\,\epsilon'>0,\;v>0,\;R>0$ and $j=1,..., l$ 
\begin{align}
\sup_{0\leq t\leq T}\int_{x_0+\epsilon'-vt}^{\infty} \;(\partial_x^{s+j}u(x,t))^2dx<c^*=c^*(\epsilon';v;x_0;j;l;s)	 \label{thm:a15}
\end{align}
 and
\begin{align}
	\int_0^T\;\int_{x_0+\epsilon'-vt}^{x_0+R-vt}\;(\partial_x^{s+l+1}u(x,t))^2dxdt<c^{**}=c^{**}(\epsilon';R;v;x_0;j;l;s).
	\label{thm:a16}
\end{align}

\end{theorem}
\vskip.1in  
\begin{remark}  In the proof of Theorem \ref{thm:1} we shall only consider the integral equation version of the IVP \eqref{gkdv} and estimates 
which hold for complex and real valued solutions.  In the case of Theorem \ref{thm:2} the proof is based on weighted energy estimates performed in the differential equations for which we need to have real-valued solutions. 
\end{remark}

\vskip.2in

Below we shall use the notation:
\begin{equation*}
\txLebnT{F}{\I}{2}=\sup_{t\in\R}\|F(t)\|_2,\;\;\;\xtLebnT{F}{\I}{2}=\sup_{x\in\R}\,(\int_{-\infty}^{\infty}|f(x,t)|^2dt)^{1/2},
\end{equation*}
with
$\txLebn{F}{\I}{2}$ and $\txLebn{F}{\I}{2}$ denoting the corresponding norms restricted to the time interval $[0,T]$.
\vskip.1in

In section 2 we shall state the necessary estimates for the proofs of Theorems \ref{thm:1} and \ref{thm:2}., which will be given in section 3. 
\section{Preliminaries}

We start this section presenting some linear estimates. The first one is concerning the sharp  (homogeneous)  version of Kato smoothing effect found in \cite{KPV}.
\begin{lemma} 
Let $\{U(t)\,:\,t\in\R\}=\{e^{-it\pa_x^3}\,:\;t\in \R\}$ denote the unitary group describing the solution of the associated linear problem to \eqref{gkdv}. Then, for any  
$f\in L^2(\R)\,$ complex or real valued
\begin{align}
	\txLebnT{U(t)f}{\I}{2}  +  \xtLebnT{\pa_x U(t)f}{\I}{2} &{} = \Big(1+\frac{1}{\sqrt{3}}\Big) \Lebn{f}{2}. \label{lest:1}
\end{align}
\end{lemma}


\vskip.1in

Next, we collect some estimates necessary to prove the main results. 

\begin{lemma} \label{lem:0}
Let $\mu>0,\;r\in\Z^+ $.  Then for any ${\md \theta} \in[0,1]$ with $(1-\theta)r\in\Z^+$
\begin{align}
	\Lebn{\Jbr{x}^{\theta \mu} \pa_x^{(1-\theta)r} f}{2} \le{}& C\Lebn{\Jbr{x}^{\mu}  f}{2}^{\theta} \Lebn{\pa_x^{r} f}{2}^{1-\theta} +L.O.T. \label{kest:1} 
	\end{align}
\end{lemma}
where the lower order terms $\,L.O.T.$ in \eqref{kest:1} are bounded by
\begin{align*}
\sum_{0\leq\beta\leq 1,\;(1-\beta)(r-1)\in\Z^+}\;\;\;\Lebn{\Jbr{x}^{\beta(\mu-1)} \pa_x^{(1-{\md \b})(r-1)} f}{2} . 
\end{align*}
\vskip.15in

\begin{proof}
The proof of this estimate follows by successive integration by parts.
\end{proof}

The inequality \eqref{kest:1} is related with the following estimates found in \cite{NP}.

\begin{lemma}[{\cite[Lemma 4]{NP}}] 
For any $a$, $b>0$ and $\ga \in (0,1)$, there exists $C>0$ such that
\begin{align*}
	\Lebn{J^{\ga a}(\Jbr{x}^{(1-\ga)b}f)}{2} \le C\Lebn{\Jbr{x}^b f}{2}^{1-\ga}\Lebn{J^a f}{2}^{\ga}, \\ 
	\Lebn{\Jbr{x}^{\ga a}(J^{(1-\ga)b}f)}{2} \le C\Lebn{J^b f}{2}^{1-\ga}\Lebn{\Jbr{x}^a f}{2}^{\ga}. 
\end{align*}
\end{lemma}

\vskip.1in

\section{Proof of the main results. }

\begin{proof}[Proof of Theorem \ref{thm:1}]
To simplify the exposition and without lost of generality we shall consider real valued functions. Let us introduce the complete metric space
\begin{align*}
	X_T = {}&\left\{ u \in C([0,T]; H^s(\R)); \right.
	 \vertiii{u}_{X_T} := \norm{u}_{L^{\I}_T H^{s}_x} + \norm{\Jbr{x}^{m} u}_{L^{\I}_T L^\I_x} \\
	&{} \qquad + \sum_{l=1}^4 \norm{\Jbr{x}^{m} \pa^l_x u}_{L^{\I}_T L^2_x} + \xtLebn{\pa_x^{s+1} u}{\I}{2} \le 5C_1 \d, \\
	{}& \qquad \left. \sup_{0 \le t \le T} \Lebn{\Jbr{x}^m (u(t)-u_0)}{\I} \le \frac{\l}2 \right\}
\end{align*}
{\md equipped with the distance function 
\[
d_{X_T}(u,v) ={} \vertiii{u-v}_{X_T}
\]
}for any $s \ge 2m+4$ with $s \in \Z^+$,  $m= \left[ \frac{1}{\a} \right] +1$. 
Notice that we have
\begin{align}
	\frac{\l}{2} \le \Jbr{x}^m |u(x,t)| \le \Jbr{x}^m |u_0(x)| + \frac{\l}{2} \label{cond:1}
\end{align}
for any $(x,t) \in \R \times [0,T] $ as long as $u \in X_T$.	
Here the constant $C_1$  will be chosen later.
Set 
\begin{align}
	\P(u(t)) = U(t)u_0 \mp \int_0^t U(t-s)(|u|^{\a}\pa_x u)(s) ds. \label{ieq:1}
\end{align}
We will prove that $\P$ is a contraction map in $X_T$. 
Let us first show that $\P$ maps from $X_T$ to itself. Recall that 
\[
	\Sobn{f}{s} \approx \xLebn{\partial_x^s f}{2} + \xLebn{f}{2}.
\]
By using \eqref{lest:1} and {\md the} Leibniz rule, one has
\begin{align}
	\begin{aligned}
	&{}\txLebn{\pa_x^{s} \P(u)}{\I}{2} + \xtLebn{\pa_x^{s+1} \P(u)}{\I}{2} \\
	\le {}&  C_0 \Lebn{\pa_x^s u_0}{2} +c\,\|\partial_x^s (|u|^{\a} \pa_x u)\|_{L_T^1L_x^2} \\
	\le {}&  C_0 \Lebn{\pa_x^s u_0}{2} + c\,\sum_{j=0}^{s} \txLebn{\pa_x^j (|u|^{\a}) \pa_x^{s+1-j} u}{1}{2} \\
	=: {}& 2C_0 \Lebn{\pa_x^s u_0}{2} + c\,\sum_{j=0}^{s} A_j.
	\end{aligned}
	\label{est:00}
\end{align}

We shall consider $A_j$. 
A use of the H\"older inequality gives us
\begin{align*}
	A_0 = {}& \txLebn{|u|^{\a} \pa_x^{s+1} u}{1}{2} \\
	\le {}& CT^{1/2} \xtLebn{\pa_x^{s+1}u}{\I}{2} \txLebn{\Jbr{x} |u|^{\a}}{\I}{\I} \Lebn{\Jbr{x}^{-1}}{2} \\
	\le {}& CT^{1/2} \xtLebn{\pa_x^{s+1}u}{\I}{2} \txLebn{\Jbr{x}^{m} u}{\I}{\I}^{\a},
\end{align*}
which yields
\begin{align}
	A_0 \le CT^{1/2} \d^{\a+1}. \label{est:0}
\end{align}
{\md The estimates for the intermediate terms $A_j$ ($2 \le j \le s-1$) can be obtained by the interpolation between the terms in $A_0$ and $A_s$. Hence,} we shall consider $A_s$.
One sees that
\begin{align*}
	A_s ={}& \txLebn{\pa_x^s (|u|^{\a}) \pa_x u}{1}{2} \\
	\le {}& CT \( \txLebn{|u|^{\a-s}|\pa_x u|^{s} \pa_x u}{\I}{2} + \cdots + \txLebn{|u|^{\a-1}|\pa_x^s u| \pa_x u}{\I}{2} \) \\
	=:{}& CT \( A_{s,s} + \cdots + A_{s,1} \).
\end{align*}
Since the middle term $A_{s,j}$ ($2 \le j \le s-1$) can be estimated by the interpolation between $A_{s,1}$ and $A_{s,s}$, it suffices to estimate $A_{s,1}$ and $A_{s,s}$. 

Using that
\begin{equation}\label{lowbound}
\langle x\rangle^{m}\geq c\,\lambda\,|u(x,t)|^{-1}
\end{equation}
and  {\md  Sobolev embedding}  we deduce that
\begin{align*}
	A_{s,1} = {}& \txLebn{|u|^{\a-1}|\pa_x^s u| \pa_x u}{\I}{2} \\
	\le {}& C\txLebn{|u|^{\a-1}|\pa_x u|}{\I}{\I} \txLebn{\pa_x^s u}{\I}{2} \\
	\le {}& C\txLebn{\langle x\rangle^{m(1-\a)}|\pa_x u|}{\I}{\I} \txLebn{\pa_x^s u}{\I}{2} \\
	\le {}& C\txLebn{\Jbr{x}^{m-1}\pa_x u}{\I}{\I} \txLebn{\pa_x^s u}{\I}{2} \\
	\le {}& C\( \txLebn{\Jbr{x}^{m}\pa_x u}{\I}{2} + \txLebn{\Jbr{x}^{m}\pa_x^2 u}{\I}{2} \) \txLebn{\pa_x^s u}{\I}{2}.
\end{align*} 

By using \eqref{lowbound}  and {\md  Sobolev embedding} again, one has 
\begin{align*}
	A_{s,s} = {}& \txLebn{|u|^{\a-s}|\pa_x u|^{s} \pa_x u}{\I}{2} \\
	\le {}& C\txLebn{\Jbr{x}^{m(s-\a)}|\pa_x u|^{s} \pa_x u}{\I}{2} \\
	\le {}& C\txLebn{\Jbr{x}^{m} \pa_x u}{\I}{\I}^{s} \txLebn{\pa_x u }{\I}{2}  \\
	\le {}& C\( \txLebn{\Jbr{x}^{m} \pa_x u}{\I}{2}^{s} + \txLebn{\Jbr{x}^{m} \pa_x^2 u}{\I}{2}^{s} \) \txLebn{\pa_x u }{\I}{2}.
\end{align*}
Therefore it holds that
\begin{align}
	\begin{aligned}
	A_s \le{}& CT \( \txLebn{\Jbr{x}^{m}\pa_x u}{\I}{2} + \txLebn{\Jbr{x}^{m}\pa_x^2 u}{\I}{2} \) \txLebn{\pa_x^s u}{\I}{2} \\
	&{}+  CT \( \txLebn{\Jbr{x}^{m} \pa_x u}{\I}{2}^{s} + \txLebn{\Jbr{x}^{m} \pa_x^2 u}{\I}{2}^{s} \) \txLebn{\pa_x u }{\I}{2} \\
	\le {}& CT(\d^2 + \d^{s+1}). 
	\end{aligned}
	\label{est:1}
\end{align} 

Combining \eqref{est:00} with \eqref{est:0} and \eqref{est:1}, we obtain
\begin{align}
	\begin{aligned}
	&{}\txLebn{\pa_x^{s} \P(u)}{\I}{2} + \xtLebn{\pa_x^{s+1} \P(u)}{\I}{2} \\
	\le{}& 2C_0 \d + CT^{1/2}\d^{\a+1} + CT\d (\d + \d^s). 
	\end{aligned}
	\label{est:4}
\end{align} 
One also sees from Sobolev embedding that
\begin{align}
	\begin{aligned}
	\txLebn{\P(u)}{\I}{2} \le{}& \xLebn{u_0}{2} + CT \txLebn{u}{\I}{\I}^{\a} \txLebn{\pa_x u}{\I}{2} \\
	\le {}& \xLebn{u_0}{2} + CT \norm{u}_{L^{\I}_T H^1_x}^{\a+1} \\
	\le {}& \d + CT\d^{\a+1}. 
	\end{aligned}
	\label{est:2}
\end{align}

Let us next consider
\begin{align*}
	\norm{\Jbr{x}^{m} \pa^l_x \P(u)}_{L^{\I}_T L^2_x}
\end{align*}
for any $l \in [1,4]$. Note that
\[
	U(-t)xU(t)f = xf - 3t \pa_x^2 f.
\]
This implies
\[
x^{j}U(t)f = U(t)(x-3t \pa_x^2)^{j} f
\]
for any $j \in \Z^+$. Therefore, by interpolation, we have
\begin{align*}
	{}\Lebn{\Jbr{x}^j U(t)f}{2} 
	\le{}& C\(\Lebn{U(t)f}{2} + \Lebn{|x|^j U(t)f}{2} \) \\
	\le{}& C\Lebn{f}{2} + C\Lebn{(x+3it \pa_x^2)^{j}f}{2} \\
	\le{}& C\Lebn{f}{2} + C \Lebn{x^{j}f}{2} +  C\,t^{j} \Lebn{\pa_x^{2j} f}{2},
\end{align*}
which yields
\begin{align}
	\Lebn{\Jbr{x}^j U(t)f}{2} \le C\Lebn{\Jbr{x}^{j}f}{2} +  C\,t^{j} \Lebn{\pa_x^{2j} f}{2} \label{evo:1}
\end{align}
for any $j \in \Z^{+}$. Thus, from \eqref{evo:1}, we deduce that
\begin{equation*}
\begin{split}
&\txLebn{\Jbr{x}^{m} \pa^l_x \P(u)}{\I}{2}\\
&\le \txLebn{\Jbr{x}^{m} U(t) \pa^l_x u_0}{\I}{2} + \txLebn{\int_0^t \Jbr{x}^{m} U(t-s) \pa^l_x \( |u|^{\a}\pa_x u\)(s) ds}{\I}{2} \\
	&\le C \Lebn{\Jbr{x}^m \pa_x^l u_0}{2} + C T^{m}\Lebn{\pa_x^{2m+l}u_0}{2} \\
	&\hskip10pt+ CT \txLebn{\Jbr{x}^m \pa_x^l (|u|^{\a}\pa_x u)}{\I}{2} + CT^{m+1} \txLebn{\pa_x^{2m+l} (|u|^{\a}\pa_x u)}{\I}{2}. 
\end{split}
\end{equation*}
Since $s \ge 2m+4$, as in the proof of \eqref{est:0} and \eqref{est:1}, the interpolation argument gives us
\begin{equation}\label{est:in}
\begin{split}
\|\pa_x^{2m+l} (|u|^{\a}\pa_x u)\|_{L^{\infty}_TL^2_x}	\le{}& C \txLebn{|u|^{\a}\pa_x u}{\I}{2} + C\txLebn{\pa_x^{s} (|u|^{\a}\pa_x u)}{\I}{2} \\
	\le{}& C\norm{u}^{\a+1}_{L^{\I}_T H^1_x} +CA_0 +CA_s \\
	\le {}& C\d^{\a+1} + C(\d^2 + \d^{s+1}).
\end{split}
\end{equation}
	
By means of interpolation once more, we deduce that
\begin{align*}
	&{}\txLebn{\Jbr{x}^m \pa_x^l (|u|^{\a}\pa_x u)}{\I}{2} \\
	\le{}& C\txLebn{\Jbr{x}^m |u|^{\a} \pa_x^{l+1} u}{\I}{2} + C\txLebn{\Jbr{x}^m \pa_x^{l}(|u|^{\a}) \pa_x u}{\I}{2} \\
	\le{}& C\txLebn{\Jbr{x}^m |u|^{\a}\pa_x^{l+1} u}{\I}{2} \\
	&{} + C\txLebn{\Jbr{x}^m |u|^{\a-l} |\pa_x u|^{l+1}}{\I}{2} + C\txLebn{\Jbr{x}^m |u|^{\a-1} |\pa_x^l u| \pa_x u}{\I}{2} \\
	=: {}& \textrm{I}_1 + \textrm{I}_2 +\textrm{I}_3.
\end{align*}
Firstly, $\textrm{I}_1$ is estimated as
\begin{align*}
	\textrm{I}_1 \le {}& C\txLebn{\Jbr{x}^m u}{\I}{\I} \txLebn{|u|^{\a-1}\pa_x^{l+1}u}{\I}{2} \\
	\le{}& C\txLebn{\Jbr{x}^m u}{\I}{\I} \txLebn{\Jbr{x}^{m-1} \pa_x^{l+1}u}{\I}{2} \\
	\le{}& C\d^2.
\end{align*}
Note that when $l=4$, we work with \eqref{kest:1}.
Further, one sees from {\md  Sobolev embedding} and $m(l-\a) < ml$ that 
\begin{align*}
	\textrm{I}_2 \le{}& C\txLebn{\Jbr{x}^{m(l-\a)} |\pa_x u|^l}{\I}{\I} \txLebn{|u|^{m} \pa_x u}{\I}{2} \\
	\le{}& C\txLebn{\Jbr{x}^m \pa_x u}{\I}{\I}^{l} \txLebn{\Jbr{x}^{m} \pa_x u}{\I}{2} \\
	\le{}& C \( \txLebn{\Jbr{x}^m \pa_x u}{\I}{2} + \txLebn{\Jbr{x}^m \pa_x^2 u}{\I}{2} \)^{l} \txLebn{\Jbr{x}^{m} \pa_x u}{\I}{2} \\
	\le{}& C\d^{l+1}.
\end{align*}
On the other hand, by using {\md  Sobolev embedding} again, we obtain
\begin{align*}
	\textrm{I}_3 \le{}& C \txLebn{\Jbr{x}^{m-1} \pa_x u }{\I}{\I} \txLebn{\Jbr{x}^{m} \pa_x^{l} u}{\I}{2} \\
	\le{}& \( \txLebn{\Jbr{x}^m \pa_x u}{\I}{2} + \txLebn{\Jbr{x}^m \pa_x^2 u}{\I}{2} \) \txLebn{\Jbr{x}^{m} \pa_x^{l} u}{\I}{2} \\
	\le{}& C\d^2.
\end{align*} 
Combining these estimates, it holds that
\begin{align*}
	\txLebn{\Jbr{x}^m \pa_x^l (|u|^{\a}\pa_x u)}{\I}{2} \le C\d^2 + C\d^{l+1}.
\end{align*}
Hence, we obtain
\begin{align}
	\norm{\Jbr{x}^{m} \pa^l_x \P(u)}_{L^{\I}_T L^2_x} \le C \d + CT\d (1 + \d^{s}) \label{est:3}
\end{align}
as long as $T \le 1$.

Next we estimate
\[
	\txLebn{\Jbr{x}^m \P(u)}{\I}{\I}.
\]
By using \eqref{evo:1} and the fact
\[
	\frac{d}{dt} U(t)u_0 = -\pa_x^3 U(t) u_0,
\]
together with Sobolev embedding, we obtain
\begin{align*}
	\le {}&\xLebn{\Jbr{x}^m (U(t)u_0 -u_0)}{\I}\\
	\le {}& \xLebn{\int_0^t \frac{d}{ds}\(\Jbr{x}^m U(s)u_0 \) ds}{\I} \\
	\le {}& \xLebn{\int_0^t \Jbr{x}^m U(s) \pa_x^3 u_0 ds}{\I} \\
	\le {}& CT \( \xLebn{\Jbr{x}^{m} \pa_x^3 u_0}{2} + \xLebn{\Jbr{x}^{m} \pa_x^4 u_0}{2} \) + CT^{m+1} \Sobn{u_0}{2m+4},
\end{align*}
which implies
\begin{align}
	\txLebn{\Jbr{x}^m (U(t)u_0 -u_0)}{\I}{\I} \le CT\d. \label{est:5}
\end{align}
if $T \le 1$.
From \eqref{evo:1} it follows that
\begin{align*}
	&{}\xLebn{\Jbr{x}^m \int_0^t U(t-s) \( |u|^{\a}\pa_x u\)(s) ds}{\I} \\
	\le{}& C\xLebn{\Jbr{x}^{m} \int_0^t U(t-s) \( |u|^{\a}\pa_x u\)(s) ds}{2} \\
	&{}+ C\xLebn{\Jbr{x}^m \int_0^t U(t-s) \pa_x \( |u|^{\a}\pa_x u\)(s) ds}{2} \\ 
	\le {}& C \int_0^t \xLebn{\Jbr{x}^{m} \( |u|^{\a}\pa_x u\)(s)}{2} ds \\
	&{}+ C \int_0^t |t-s|^{m} \xLebn{ \pa_x^{2m} \( |u|^{\a}\pa_x u\)(s)}{2} ds \\
	&{}+ C \int_0^t \xLebn{\Jbr{x}^{m} \pa_x \( |u|^{\a}\pa_x u\)(s)}{2} ds \\
	&{}+ C \int_0^t |t-s|^{m} \xLebn{ \pa_x^{2m+1} \( |u|^{\a}\pa_x u\)(s)}{2} ds \\
	\le {}& CT \( \txLebn{\Jbr{x}^{m} |u|^{\a}\pa_x u}{\I}{2} + \txLebn{\Jbr{x}^{m} \pa_x \(|u|^{\a}\pa_x u\)}{\I}{2} \) \\
	&{} + CT^{m} \( \txLebn{ \pa_x^{2m} \(|u|^{\a}\pa_x u\)}{\I}{2} + T \txLebn{ \pa_x^{2m+1} \(|u|^{\a}\pa_x u\)}{\I}{2} \). 
\end{align*}
Also, {\md Sobolev embedding} provides that
\begin{align*}
	&{} \txLebn{\Jbr{x}^{m} \pa_x (|u|^{\a}\pa_x u)}{\I}{2} \\
	\le{}& C\txLebn{\Jbr{x}^{m} |u|^{\a-1}(\pa_x u)^2}{\I}{2}+  \txLebn{\Jbr{x}^{m}|u|^{\a}\pa_x^2 u}{\I}{2} \\
	\le{}& C\txLebn{\Jbr{x}^{2m-\a m} (\pa_x u)^2}{\I}{2}+  \txLebn{\Jbr{x}^{m}|u|^{\a}\pa_x^2 u}{\I}{2} \\
	\le{}& C\txLebn{\Jbr{x}^{m} \pa_x u}{\I}{\I} \txLebn{\Jbr{x}^{m} \pa_x u}{\I}{2}+  C\|u\|^{\a}_{L^{\infty}_TL^{\infty}_x}\txLebn{\Jbr{x}^{m}\pa_x^2 u}{\I}{2} \\
	\le{}& C(\delta^{1+{\md \a}}+\d^2).
\end{align*}
Further, a similar computation as in \eqref{est:in} gives us 
\begin{align*}
	\txLebn{\pa_x^{\b}(|u|^{\a} \pa_x u)}{\I}{2} \le C\d^{\a+1} + C(\d^2 + \d^{s+1}) \end{align*}
for some $\b \in [0,s]$. Therefore it follows from these estimates that
\begin{align}
	\begin{aligned}
	&{}\txLebn{\Jbr{x}^m \int_0^t U(t-s) \( |u|^{\a}\pa_x u\)(s) ds}{\I}{\I} \\
	\le{}& CT\d^{2} + CT^{m+1} (\d^{\a+1} + \d^{s+1}). 
	\end{aligned}
	\label{est:6}
\end{align}
Combining \eqref{est:5} with \eqref{est:6}, we see that
\begin{align}
	\begin{aligned}
	\txLebn{\Jbr{x}^m \P(u)}{\I}{\I} 
	\le{}& \xLebn{\Jbr{x}^{m}u_0}{\I} + \txLebn{\Jbr{x}^m\( U(t)u_0-u_0\)}{\I}{\I} \\
	&{}+ \txLebn{\Jbr{x}^m \int_0^t U(t-s)\(|u|^{\a}\pa_x u\)(s) ds}{\I}{\I} \\
	\le{}& \d + CT(\d+\d^{s+1})
	\end{aligned}
	\label{est:7}
\end{align}
whenever $T \le 1$.
Therefore, by using \eqref{est:4} and \eqref{est:2} together with \eqref{est:3} and \eqref{est:7}, one establishes that
\begin{align*}
	\vertiii{u}_{X_T} 
\le{}& 2C_0 \d + CT^{1/2}\d^{\a+1} + CT\d (\d + \d^s) + \d + CT\d^{\a+1} \\
	&{}+ C \d + CT\d (1 + \d^{s}) + \d + CT(\d+\d^{s+1}) \\
	\le{}& 4C_1\d + CT^{1/2}\d^{\a+1} + CT\d (1 + \d^{s}) \\
	\le{}& 5C_1\d,
\end{align*}
where $C_1 = \max(2C_0, C, 1)$ as long as $T = T(\a, \d, s)$ is small enough.

Moreover, as in the proof of \eqref{est:7}, it holds that
\begin{align*}
	&{}\sup_{0 \le t \le T} \Lebn{\Jbr{x}^m (\P(u(t))-u_0)}{\I} \\
	\le{}& \xLebn{\Jbr{x}\( U(t)u_0-u_0\)}{\I} + \xLebn{\Jbr{x}^m \int_0^t U(t-s)\(|u|^{\a}\pa_x u\)(s) ds}{\I} \\
	\le {}& CT(\d + \d^{1+s}) \le \frac{\l}2 
\end{align*}
if $T = T(\d, s, \l)$ is sufficiently small. Thus, $\P(u) \in X_T$ holds.

Let us show $\P$ is a contraction map in $X_T$. By using \eqref{lest:1}, we first estimate
\begin{align}
	\begin{aligned}
	&{}\txLebn{\pa_x^{s} \( \P(u)-\P(v)\)}{\I}{2} + \xtLebn{\pa_x^{s+1} \( \P(u)-\P(v)\)}{\I}{2} \\
	\le {}& \sum_{j=0}^{s} \txLebn{\pa_x^j (|u|^{\a}) \pa_x^{s+1-j} u - \pa_x^j (|v|^{\a}) \pa_x^{s+1-j} v}{1}{2} \\
	=: {}& \sum_{j=0}^{s} B_j.
	\end{aligned}
	\label{est:d1}
\end{align}
Similarly to the above, by the interpolation argument, it suffices to deal with $B_0$ and $B_s$. Here we observe that
\begin{equation}\label{est:d2}
	||u|^{\a} - |v|^{\a}|=\alpha (\theta\,|u|+ (1-\theta)|v|)^{\a-1})(|u|-|v|) \leq c\,\langle x\rangle^{m(1-\a)}|u-v|
\end{equation}
for $\t\in (0,1)$ and any $\a \in \R$. Together with \eqref{est:d2}, a similar computation as in \eqref{est:0} shows that 
\begin{equation*}\label{est:d2a}
\begin{split}
	B_0 =& \txLebn{|u|^{\a}\pa_x^{s+1}u - |v|^{\a}\pa_x^{s+1}v}{1}{2} \\
	\le{}& C \txLebn{|u|^{\a} \pa_x^{s+1}(u -v)}{1}{2} + C\txLebn{(|u|^{\a-1} + |u|^{\a-1})|u-v|\pa_x^{s+1}v}{1}{2} \\
	\le{}& CT^{1/2} \txLebn{\Jbr{x}^m u}{\I}{\I} \xLebn{\Jbr{x}^{-1}}{2} \xtLebn{\pa_x^{s+1}(u-v)}{\I}{2} \\
	&+ CT^{1/2} \txLebn{\Jbr{x}^m (u-v)}{\I}{\I} \xLebn{\Jbr{x}^{-1}}{2} \xtLebn{\pa_x^{s+1}v}{\I}{2}. 
\end{split}
\end{equation*}
This implies
\begin{align}
	B_0 \le CT^{1/2}\d d_{X_T}(u,v). \label{est:d3}
\end{align}
On the other hand, we see from \eqref{est:d1} that
\begin{align*}
	B_s \le{}& CT \txLebn{|u|^{\a-s}(\pa_x u)^{s+1} - |v|^{\a-s}(\pa_x v)^{s+1}}{\I}{2} \\
	&{} + \dots \\
	&{} + CT \txLebn{|u|^{\a-1}\pa_x^s u \pa_x u - |v|^{\a-1} \pa_x^s v  \pa_x v}{\I}{2} \\
	=:{}& CT(B_{s,1} + \dots + B_{s,s}). 
\end{align*}
Similarly to \eqref{est:1}, the middle terms $B_{s,j}$ ($2 \le j \le s-1$) can be estimated by the interpolation between $B_{s,1}$ and $B_{s,s}$, so it suffices to 
estimate $B_{s,1}$ and $B_{s,s}$. 
By using {\md  Sobolev embedding} and \eqref{est:d2}, one has that
\begin{align*}
	B_{s,1} \le{}& \txLebn{|u|^{\a-s} ( (\pa_x u)^{s+1} - (\pa_x v)^{s+1})}{\I}{2}\\
& + \txLebn{( |u|^{\a-s} - |u|^{\a-s}) (\pa_x v)^{s+1}}{\I}{2} \\
	\le{}& \( \txLebn{\Jbr{x}^{m} \pa_x u}{\I}{2}^{s} + \txLebn{\Jbr{x}^{m} \pa_x v}{\I}{2}^{s} \) \txLebn{\pa_x (u-v) }{\I}{2} \\
	&{}+ \( \txLebn{\Jbr{x}^{m} \pa_x v}{\I}{2}^{s} + \txLebn{\Jbr{x}^{m} \pa_x^2 v}{\I}{2}^{s} \)\\
& \;\times \txLebn{\pa_x v}{\I}{2} \txLebn{\Jbr{x}^m (u-v)}{\I}{\I} \\
	\le{}& \d^{s} (1+\d) d_{X_T}(u,v).
\end{align*}
We also obtain that
\begin{align*}
	B_{s,s} ={}& \txLebn{|u|^{\a-1}\pa_x^s u \pa_x u - |v|^{\a-1} \pa_x^s v  \pa_x v}{\I}{2} \\
	\le{}& \txLebn{|u|^{\a-1} \pa_x^s u \pa_x (u-v) }{\I}{2} + \txLebn{|u|^{\a-1} \pa_x^s (u-v) \pa_x v}{\I}{2} \\
	&{}+ \txLebn{(|u|^{\a-1} - |v|^{\a-1}) \pa_x^s v  \pa_x v}{\I}{2} \\
	=:{}& E_1 + E_2 + E_3.
\end{align*}
To estimate $E_1$ and $E_2$ we follow an argument similar to the one used in \eqref{est:1}, so that one obtains
\begin{align*}
	E_1 +E_2 \le{}& C \( \txLebn{\Jbr{x}^{m}\pa_x (u-v)}{\I}{2} + \txLebn{\Jbr{x}^{m}\pa_x^2 (u-v)}{\I}{2} \) \txLebn{\pa_x^s u}{\I}{2} \\
	&{}+  C \( \txLebn{\Jbr{x}^{m} \pa_x u}{\I}{2} + \txLebn{\Jbr{x}^{m} \pa_x^2 u}{\I}{2} \) \txLebn{\pa_x^s (u-v) }{\I}{2} \\
	\le{}& C\d d_{X_T}(u,v). 
\end{align*}
From {\md Sobolev embedding} and \eqref{est:d2}, $E_3$ is estimated as
\begin{align*}
	E_3 ={}& \txLebn{(|u|^{\a-1} - |v|^{\a-1}) \pa_x^s v  \pa_x v}{\I}{2} \\
	\le{}& C\txLebn{(|u|^{\a-2} + |v|^{\a-2})|u-v| \pa_x^s v  \pa_x v}{\I}{2} \\
	\le{}& C\txLebn{\Jbr{x}^{2m-1} |u-v| \pa_x^s v  \pa_x v}{\I}{2} \\
	\le{}& C\txLebn{\Jbr{x}^{m} (u-v) }{\I}{\I} \txLebn{\pa_x^s v}{\I}{2} \\
	&{} \; \times \( \txLebn{\Jbr{x}^{m} \pa_x v}{\I}{2} + \txLebn{\Jbr{x}^{m} \pa_x^2 v}{\I}{2}\) \\
	\le{}& C\d^2 d_{X_T}(u,v).
\end{align*}
Thus, since
\begin{align*}
	B_{s,s} \le E_1 + E_2 + E_3 \le C\d(1+\d) d_{X_T}(u,v),
\end{align*}
we have
\begin{align}
	B_s \le CT\d(1+\d^s) d_{X_T}(u,v). \label{est:d4}
\end{align}
Combining \eqref{est:d1} with \eqref{est:d3} and \eqref{est:d4}, it holds that
\begin{align}
	\begin{aligned}
	&{}\txLebn{\pa_x^{s} \( \P(u)-\P(v)\)}{\I}{2} + \xtLebn{\pa_x^{s+1} \( \P(u)-\P(v)\)}{\I}{2} \\
	\le {}& CT^{1/2}\d d_{X_T}(u,v) + CT\d(1+\d^s) d_{X_T}(u,v).
	\end{aligned}
	\label{est:d5}
\end{align}
Moreover, by using \eqref{est:d2} and  Sobolev embedding, together with \eqref{lowbound}, we deduce that
\begin{equation*}
\begin{split}
	&{}\txLebn{\P(u) -\P(v)}{\I}{2} \\
	 \le{}& CT\txLebn{|u|^{\a} \pa_x (u-v)}{\I}{2} + CT \txLebn{(|u|^{\a} - |v|^{\a}) \pa_x v}{\I}{2} \\
	\le{}& CT\norm{u}_{L^{\I}_T H^1_x}^{\a} \txLebn{\pa_x (u-v)}{\I}{2} + CT \txLebn{(|u|^{\a-1} + |v|^{\a-1})|u-v| \pa_x v}{\I}{2} \\
	\le{}& CT\norm{u}_{L^{\I}_T H^1_x}^{\a} \txLebn{\pa_x (u-v)}{\I}{2} + CT \txLebn{\Jbr{x}^m |u-v|}{\I}{\I} \txLebn{\pa_x v}{\I}{2},
\end{split}
\end{equation*}
which yields
\begin{align}
	\txLebn{\P(u) -\P(v)}{\I}{2} \le CT(\d+\d^{\a})d_{X_T}(u,v). \label{est:d6}
\end{align}

Let us estimate 
\[
	\txLebn{\Jbr{x}^m \( \P(u)-\P(v)\)}{\I}{\I}.
\]
Arguing as in \eqref{est:6}, by means of the Sobolev embedding and \eqref{evo:1},
we first compute
\begin{align}
	\begin{aligned}
	&{}\txLebn{\Jbr{x}^m (\P(u)-\P(v))}{\I}{\I} \\
	\le{}& C\txLebn{\int_0^t \Jbr{x}^{m} U(t-s)(|u|^{\a}\pa_x u - |v|^{\a}\pa_x v)(s) ds}{\I}{2} \\
	&{}+ C\txLebn{\int_0^t \Jbr{x}^m U(t-s) \pa_x (|u|^{\a}\pa_x u - |v|^{\a}\pa_x v)(s) ds}{\I}{2} \\
	\le{}& C \txLebn{\Jbr{x}^m (|u|^{\a}\pa_x u - |v|^{\a}\pa_x v)}{1}{2} \\
	&{}+ C \txLebn{\Jbr{x}^m \pa_x (|u|^{\a}\pa_x u - |v|^{\a}\pa_x v)}{1}{2} \\
	&{}+ CT^{m} \txLebn{\pa_x^{2m} (|u|^{\a}\pa_x u - |v|^{\a}\pa_x v)}{1}{2} \\
	&{}+ CT^m \txLebn{\pa_x^{2m+1} (|u|^{\a}\pa_x u - |v|^{\a}\pa_x v)}{1}{2} \\
	=:{}& J_1 + J_2 +J_3 +J_4.
	\end{aligned}
	\label{est:d7}
\end{align}
Using interpolation the terms $J_3$ and $J_4$ can be handled applying the same argument as in \eqref{est:d5} and \eqref{est:d6}. Thus
\begin{align}
	J_3 + J_4 \le{}& CT^{1/2}\d d_{X_T}(u,v) + CT\d(1+\d^s) d_{X_T}(u,v) \label{est:d8}
\end{align}
whenever $T \le 1$.
Further, it holds that
\begin{align*}
	J_1 \le{}& 
	CT \txLebn{\Jbr{x}^m u}{\I}{\I} \txLebn{\Jbr{x}^{m-1} \pa_x (u-v)}{\I}{2} \\
	&{}+ CT \txLebn{\Jbr{x}^m (u-v)}{\I}{\I} \txLebn{\Jbr{x}^{m-1} \pa_x v}{\I}{2},
\end{align*}
which implies that
\begin{align}
	J_1 \le CT\d d_{X_T}(u,v). \label{est:d9}
\end{align}
Next, we estimate $J_2$, indeed, 
\begin{align*}
	J_2 \le{}& C \txLebn{\Jbr{x}^m  (|u|^{\a-1}(\pa_x u)^2 - |v|^{\a-1}(\pa_x v)^2)}{1}{2} \\
	&{}+ C \txLebn{\Jbr{x}^m (|u|^{\a}\pa_x^2 u - |v|^{\a}\pa_x^2 v)}{1}{2} \\
	=:{}& J_{2,1} +J_{2,2}.
\end{align*}
Hence, it comes from \eqref{est:d2} and {\md  Sobolev embedding} that
\begin{align*}
	J_{2,1} \le{}& C \txLebn{\Jbr{x}^m  |u|^{\a-1}((\pa_x u)^2 -(\pa_x v)^2)}{1}{2} \\
	&{}+ C \txLebn{\Jbr{x}^m  (|u|^{\a-1}-|v|^{\a-1}) (\pa_x v)^2}{1}{2} \\
	\le{}& CT\d(1+\d)d_{X_T}(u,v).
\end{align*}
On the other hand, we obtain employing \eqref{est:d2} once again and the argument leading to \eqref{est:d3} that
\begin{align*}
	J_{2,2} \le{}& C \txLebn{\Jbr{x}^m (|u|^{\a} - |v|^{\a}) \pa_x^2 v}{1}{2}  + C\txLebn{\Jbr{x}^m |u|^{\a} \pa_x^2 (u-v)}{1}{2} \\
	\le{}& CT\d d_{X_T}(u,v).
\end{align*}
Hence one establishes that
\begin{align}
	J_{2} \le CT\d(1+\d)d_{X_T}(u,v). \label{est:d10}
\end{align}
Thus, collecting \eqref{est:d7}, \eqref{est:d8}, \eqref{est:d9} and \eqref{est:d10}, it follows that
\begin{align}
	\begin{aligned}
	&{}\txLebn{\Jbr{x}^m (\P(u)-\P(v))}{\I}{\I} \\
	\le{}& CT^{1/2}\d d_{X_T}(u,v) + CT\d(1+\d^s) d_{X_T}(u,v)
	\end{aligned}
	\label{est:d11}
\end{align}
as long as $T \le 1$.
Finally, we turn to consider
\[
	\txLebn{\Jbr{x}^m \pa_x^l (\P(u)-\P(v))}{\I}{2}
\]
for any $l \in [1,4]$. By using \eqref{evo:1}  one has
\begin{align}
	\begin{aligned}
	&{}\txLebn{\Jbr{x}^m \pa_x^l (\P(u)-\P(v))}{\I}{2} \\
	\le{}& C\txLebn{\Jbr{x}^m \pa_x^l (|u|^{\a}\pa_x u - |v|^{\a}\pa_x v)}{1}{2} \\
	&{}+ CT^{m} \txLebn{\pa_x^{2m+l} (|u|^{\a}\pa_x u - |v|^{\a}\pa_x v)}{1}{2} \\
	=:{}& F_1 +F_2.
	\end{aligned}
	\label{est:d12}
\end{align}
$F_2$ can be estimated following the argument in \eqref{est:d8}. Hence we have
\begin{align}
	F_2 \le{}& CT^{1/2+m}\d d_{X_T}(u,v) + CT^{m+1}\d(1+\d^s) d_{X_T}(u,v) \label{est:d13}.
\end{align}
Let us estimate $F_1$. A use of the triangle inequality tells us that
\begin{align*}
	F_1 \le{}& C\txLebn{\Jbr{x}^m ( |u|^{\a-l}(\pa_x u)^{l+1} - |v|^{\a-l}(\pa_x v)^{l+1} )}{1}{2} \\
	&{}+ C\txLebn{\Jbr{x}^m ( |u|^{\a-1}\pa_x^{l} u \pa_x u - |v|^{\a-1}\pa_x^{l} v \pa_x v)}{1}{2} \\
	&{}+ C\txLebn{\Jbr{x}^m ( |u|^{\a} \pa_x^{l+1} u - |v|^{\a} \pa_x^{l+1} v)}{1}{2} \\
	=:&{} F_{1,1} +F_{1,2} +F_{1,3}.
\end{align*}
It comes from \eqref{cond:1}, \eqref{est:d2} and {\md  Sobolev embedding} that
\begin{align*}
	F_{1,1} \le{}& C \txLebn{\Jbr{x}^m  |u|^{\a-l}((\pa_x u)^{l+1} -(\pa_x v)^{l+1})}{1}{2} \\
	&{}+ C \txLebn{\Jbr{x}^m  (|u|^{\a-l}-|v|^{\a-l}) (\pa_x v)^{l+1}}{1}{2} \\
	\le{}& CT\d^l (1+\d)d_{X_T}(u,v).
\end{align*}
On the other hand, combining \eqref{cond:1} and \eqref{est:d2}, we obtain
\begin{align*}
	F_{1,3} \le{}& C \txLebn{\Jbr{x}^m (|u|^{\a} - |v|^{\a}) \pa_x^{l+1} v}{1}{2}  + C\txLebn{\Jbr{x}^m |u|^{\a} \pa_x^{l+1} (u-v)}{1}{2} \\
	\le{}& CT\d d_{X_T}(u,v).
\end{align*}
Observe that if $l=4$ we use \eqref{kest:1} to estimate $F_{1,3}$. 
As for $F_{1,2}$, one sees from \eqref{cond:1} and \eqref{est:d2} that
\begin{align*}
	F_{1,2} 
	\le{}& CT\txLebn{\Jbr{x}^m |u|^{\a-1} \pa_x^{l} u \pa_x (u-v) }{\I}{2} \\
	&{}+ CT\txLebn{\Jbr{x}^m |u|^{\a-1} \pa_x^{l} (u-v) \pa_x v}{\I}{2} \\
	&{}+ CT\txLebn{\Jbr{x}^m (|u|^{\a-2} - |v|^{\a-2})|u-v| \pa_x^{l} v  \pa_x v}{\I}{2} \\
	\le{}& CT\d d_{X_T}(u,v).
\end{align*}
Combining these estimates, it holds that
\begin{align}
	F_1 \le CT\d d_{X_T}(u,v) + CT\d^l (1+\d)d_{X_T}(u,v). \label{est:d14}
\end{align}
Therefore, by unifying \eqref{est:d12}, \eqref{est:d13} and \eqref{est:d14}, we establish
\begin{align}
	\begin{aligned}
	&{}\txLebn{\Jbr{x}^m \pa_x^l (\P(u)-\P(v))}{\I}{2} \\
\le{}& CT^{1/2}\d d_{X_T}(u,v) + CT\d(1+\d^s) d_{X_T}(u,v)
	\end{aligned}
	\label{est:d15}
\end{align}
if $T \le 1$.
In conclusion, combining \eqref{est:d5} with \eqref{est:d6}, \eqref{est:d11} and \eqref{est:d15}, we see that
\begin{align*}
	d_{X_T}(\P(u),\P(v)) \le \frac12 d_{X_T}(u, v)
\end{align*}
as long as $T = T(\d, s)$ is sufficiently small. This tells us $\P$ is a contraction map in $X_T$, that is, \eqref{gkdv} has a unique local solution in $X_T$. The reminder of the proof is standard, so we omit the detail. This completes the proof.
\end{proof}
\begin{proof}[Proof of Theorem \ref{thm:2}]

Without loss of generality we shall assume $\,x_0=0$. First, we introduce a two parameter family of cut-off functions
$\,\chi_{\epsilon, b}$ : for any $\epsilon>0, \,b\geq 5\epsilon$
\begin{equation*}
\chi_{\epsilon,b}(x)=
\begin{cases}
\begin{aligned}
& 0,\;\;\;\;x<\epsilon,\\
&1,\;\;\;\;x>b-\epsilon,
\end{aligned}
\end{cases}
\end{equation*}
with
\begin{equation*}
\begin{aligned}
\begin{cases}
&\chi'_{\epsilon,b}(x)\geq 0,\;\;\;\supp(\chi_{\epsilon,b})\subseteq [\epsilon,\infty),\;\;\;\supp(\chi'_{\epsilon,b})\subseteq [\epsilon,b-\epsilon],\\
\\
&\chi'_{\epsilon,b}(x)\geq \frac{1}{b-4\epsilon},\;x\in[2\epsilon,b-2\epsilon],\\
\\
&\chi_{\epsilon/2,b}(x)\geq c_{\epsilon,b}(\chi_{\epsilon,b}(x)+\chi'_{\epsilon,b}(x)),\;\;\;\;\;x\in\R.
\end{cases}
\end{aligned}
\end{equation*}
\vskip.1in
By formally taking the $s+j,$ ($j=1,...,l$)  derivative of the equation in \eqref{gkdv}, multiplying the result by
$\,\partial_x^{s+j}u(x,t)\,\chi_{\epsilon,b}(x+vt)$ for arbitrary $\epsilon>0, v>0$ and $\,b\geq 5\epsilon$ and integrating the result in the space variable, after some integration by parts, it follows that
\begin{equation}\label{A3}
\begin{aligned}
&\frac12\frac{d}{dt}\int (\partial_x^{s+j} u)^2(x,t)\chi_{\epsilon,b}(x+vt)\,dx\\
&-\underset{A_1}{ \underbrace{{\md \frac{v}{2}} \int  (\partial_x^{s+j} u)^2(x,t)\chi_{\epsilon,b}'(x+vt)\,dx}}\\
&+\underset{A_2}{\underbrace{\frac32\int  (\partial_x^{s+j+1} u)^2(x,t)\chi_{\epsilon,b}'(x+vt)\,dx}}
\\
&-\underset{A_3}{ \underbrace{\frac12\int  (\partial_x^{s+j}u)^2(x,t)\chi_{\epsilon,b}'''(x+vt)\,dx}}\\
&{\md \pm} \underset{A_4}{ \underbrace{\int \partial_x^{s+j}(|u|^{\alpha}\partial_x u)\partial_x^{s+j}u(x,t)\,\chi_{\epsilon,b}(x+vt)\,dx}}=0.
\end{aligned}
\end{equation}
{\md Note that the above formal computation is justified by arguing as in \cite[Section 3]{ILP}.}  
The idea is to use the formula \eqref{A3} and induction argument in $ l\in\Z^+$ to establish \eqref{thm:a15} and \eqref{thm:a16}.
\vskip.1in
\underline {Case: $l=1$}: We observe that the term $A_2$ in \eqref{A3} is positive. Also, after integration in the time interval $[0,T]${\md,} the terms 
$A_1$ and $A_3$ are bounded by using the second statement in \eqref{thm:13}.  Hence, one just needs to consider the contributions of the term 
$A_4$ in \eqref{A3}.

 Thus, we write
 \begin{equation}
 \begin{aligned}
 \label{A4}
 \partial_x^{s+1}(|u|^{\alpha}\partial_xu)=&|u|^{\alpha}\partial_x^{s+2}u+2\alpha |u|^{\alpha-1} \partial_xu\partial^{s+1}_xu+\dots\\
 &+c_{\alpha,s}|u|^{\alpha-(s+1)}(\partial_xu)^{s+1}\partial_xu.
 \end{aligned}
 \end{equation}
 
 By integration by parts{\md,} one sees that
 \begin{equation} \label{A5}
 \begin{split}
 \int |u|^{\alpha} &\partial_x^{s+2}u\partial_x^{s+1}u(x,t)\,\chi_{\epsilon,b}(x+vt)\,dx\\
=&
 -\frac{\alpha}{2} \int {\md |u|^{\alpha-2}u} \partial_xu(\partial_x^{s+1}u(x,t))^2\,\chi_{\epsilon,b}(x+vt)\,dx\\
 &\;-\frac{1}{2} \int|u|^{\alpha}(\partial_x^{s+1}u(x,t))^2\,\chi'_{\epsilon,b}(x+vt)\,dx={\md:} B_1+B_2.
 \end{split}
 \end{equation}
 Since for $x\in\R$ 
 \begin{equation}
  \label{A5a}
 |u|^{\alpha-1}|\partial_xu|(x,t)\leq c \|\Jbr{x}^{m(1-\alpha)}\partial_xu(t)\|_{\infty},
 \end{equation}
 \noindent combining the facts that after integration in the time interval $[0,T]${\md,}
 \begin{equation}
  \label{A5b}
 \int(\partial_x^{s+1}u(x,t))^2\,\chi'_{\epsilon,b}(x+vt)\,dx
\end{equation}
is bounded {\md with}
\begin{equation}
 \label{A5c}
\sup_{0\leq t \leq T}\|u(t)||_{\infty}<\infty,
\end{equation}
one can control the contributions of the terms $B_1$ and $B_2$ in \eqref{A5} for $l=1$. The argument to estimate the second term in the right hand side (r.h.d.) of \eqref{A4} is similar to that already done for $B_{\md 1}$ in \eqref{A5}. So it remains to consider the third term in the r.h.s. of \eqref{A4}. For this we write
\begin{equation}
\label{A6}
\begin{aligned}
&|\int |u|^{\alpha-(s+1)}(\partial_xu)^{s+1}\partial_xu\partial_x^{s+1}u(x,t)\,\chi_{\epsilon,b}(x+vt)\,dx|\\
&\leq c {\color{black} \int \Jbr{x}^{m(s+1-\alpha)}|\partial_x u|^{s+2} |\partial_x^{s+1}u(x,t)|\,|\chi_{\epsilon,b}(x+vt)|\,dx} \\
&\leq c \|  \Jbr{x}^{m}\partial_xu\|_{\infty}^{s+1} \|\partial_xu\|_2 (\int(\partial_x^{s+1}u(x,t))^2\,\chi_{\epsilon,b}(x+vt)\,dx)^{1/2},
\end{aligned}
\end{equation}
whose contribution can be bounded after using Gronwall's inequality in \eqref{A3}. 

This basically completes the proof of the case $l=1$ in the induction argument, i.e. \eqref{thm:a15} and \eqref{thm:a16}  with $j=l=1$. We remark that the terms omitted in \eqref{A4} (and in the proof of Theorem \ref{thm:1}) can be handled as they will be done in the next step, see \eqref{A13}-\eqref{A24}.
\vskip.1in
Now assuming the result \eqref{thm:a15} and \eqref{thm:a16}  for $l=r$ we shall prove it for $l=r+1$.

Thus, we consider the identity \eqref{A3} with $j=l+1$. As before the term $A_2$ is positive and the term $A_1$ and $A_3$ are bounded by the hypothesis of induction \eqref{thm:a16}  with  $l=r$ for  an appropriate value of $\epsilon'$ and $R$ there. Therefore, it remains to consider $A_4$ with $\,j=l+1$. 

First, we notice that by Theorem \ref{thm:1} one has
\begin{equation*} 
\Jbr{x}^m\,\partial_x^ku\in L^2(\R),\;\;\;k=1,\dots,4,
\end{equation*}
and by hypothesis of induction for any $\epsilon>0$ and $b\geq 5\epsilon$
\begin{equation}
\label{A8}
\partial_x^{s+j}u(x,t)\varphi_{\epsilon,b}(x+vt)\in L^2(\R),\;\;\;j=1,\dots, r,\;\;\;\varphi_{\epsilon,b}(x)=\sqrt{\chi_{\epsilon,b}(x)}.
\end{equation}
 Using that 
 \begin{equation}
 \label{A9}
 \Jbr{x}^m\,\chi_{\epsilon,b}'(x+vt)\leq c \chi_{\epsilon/2,b}(x+vt),\;\;\;\;\;\;c=c(m;v;t;\epsilon;b),
\end{equation}
successive integration by parts show that for any $\theta\in[0,1]$ with \newline $\theta k+(1-\theta)(s+r)\in\Z,\;k=1,2,3,4$
\begin{equation}
\label{A10}
\Jbr{x}^{\theta m}\partial_x^{\theta k+(1-\theta)(s+r)}u(x,t)\varphi_{\epsilon,b}(x+vt)\in L^2(\R),
\end{equation}
and by Sobolev embedding
\begin{equation}
\label{A11}
\Jbr{x}^{\theta m}\partial_x^{\theta k+(1-\theta)(s+r)-1}u(x,t)\chi_{\epsilon,b}(x+vt)\in L^{\infty}(\R),
\end{equation}
since the one already has the estimates for the lower order terms. 

Thus, we need to estimate the term \eqref{A4}  in \eqref{A3} 
\begin{equation}
\label{A12}
\int \partial_x^{s+r+1}(|u|^{\alpha}\partial_x u)\partial_x^{s+r+1}u(x,t)\,\chi_{\epsilon,b}(x+vt)\,dx,
\end{equation}
so, we write
 \begin{equation}
 \begin{aligned}
 \label{A13}
 \partial_x^{s+r+1}(|u|^{\alpha}\partial_xu) = &\, |u|^{\alpha}\partial_x^{s+r+2}u+2\alpha |u|^{\alpha-1} \partial_xu\partial^{s+r+1}_xu\\
 &+D_{s+r+1}.
 \end{aligned}
 \end{equation}
The argument to handle the contribution of the first two terms in the r.h.s. of \eqref{A13} is similar to that described in \eqref{A5}-\eqref{A6} so it will be omitted. Then it remains to consider the contribution of $D_{s+r+1}{\md=:}D$ in \eqref{A13} when it is inserted in the term $A_4$ in \eqref{A5} with $j=r+1$. We observe that $D$ is the sum of terms which are product of factors involving at derivatives of order at most $s+r$. In fact, one has that
\begin{equation}
\label{A15}
D=\sum_{n+\beta_0=s+r+1\atop 0\leq \beta_0\leq s+{\md r}-1}\;\sum_{\beta_1+\dots+\beta_n=n>1\atop 1\leq \beta_1,\dots,\beta_n\leq s+r}c_{\vec{\beta}}|u|^{\alpha-n}\,\partial_x^{\beta_1}u\dots\partial_x^{\beta_n}u\,\partial_x^{\beta_0+1}u.
\end{equation} 

When {\md $D$} is inserted in the term $A_4$ in \eqref{A3}{\md, this} yields an expression of the form
\begin{equation}
\label {A16}
\int |u|^{\alpha-n}\,\partial_x^{\beta_1}u\dots\partial_x^{\beta_n}u\,\partial_x^{\beta_0+1}u\partial_x^{s+r+1}u(x,t)\,\chi_{\epsilon,b}(x+vt)\,dx
\end{equation}
which are bounded in absolute value by
\begin{equation}
\label {A17}
\int \Jbr{x}^{m(n-\alpha)} {\md |} \partial_x^{\beta_1}u\dots\partial_x^{\beta_n}u\,\partial_x^{\beta_0+1}u\partial_x^{s+r+1}u(x,t){\md |}\,\chi_{\epsilon,b}(x+vt)\,dx.
\end{equation}
Consequently, it suffices to see that
\begin{equation}
\label{A18}
E_1= \|\Jbr{x}^{m(n-\alpha)}\,\partial_x^{\beta_1}u\dots\partial_x^{\beta_n}u\,\partial_x^{\beta_0+1}u\,\varphi(x+vt)\|_2
\end{equation}
is controlled by a product of the terms in \eqref{A10}-\eqref{A11} which have been already bounded in the previous {\md case} $ l=r$.
By an appropriate modification of the parameters in $\chi$, see \eqref{A8}-\eqref{A9}, one gets
\begin{equation}
\label{A19}
\begin{aligned}
E_1\leq 
\prod_{j=1}^n & \|\Jbr{x}^{m\theta_j}\,\partial_x^{(1-\theta_j)(s+r)+\theta_jk_j-1}u\,\tilde{\chi}\|_{\infty}\\
&{\md \times} \|\Jbr{x}^{m\theta_0}\,\partial_x^{(1-\theta_0)(s+r)+\theta_0k_j}u\,\tilde{\varphi}\|_{2}{\md,} \\
\end{aligned}
\end{equation}
where
\begin{equation}
\label{A20}
\begin{aligned}
&k_0,\;k_j\in\{1,\dots,4\},\;\;\;j=1,\dots, n,\\
&\beta_j=(1-\theta_j)(s+r)+\theta_jk_j-1,\;\;j=1,\dots,n,\\
& \beta_0=(1-\theta_0)(s+r)+\theta_0k_0-1.
\end{aligned}
\end{equation}

We shall complete the proof by establishing that
\begin{equation}
\label{A21}
\theta_0+\theta_1+\dots+\theta_n\geq n-\alpha.
\end{equation}
After some computations from \eqref{A20} one finds that
\begin{equation}
\label{A22}
{\md n+2}=(n-(\theta_0+\theta_1+\dots+\theta_n))(s+r)+\theta_0k_0+\theta_1k_1+\dots+\theta_nk_n,
\end{equation}
which implies that
\begin{equation}
\label{A23}
\begin{aligned}
\theta_0+\theta_1+\dots+\theta_n&=n+\frac{\theta_0k_0+\theta_1k_1+\dots+\theta_nk_n-{\md (n+2)}}{r+s}\\
&\geq n+\frac{\theta_0+\theta_1+\dots+\theta_n-{\md (n+2)}}{r+s},
\end{aligned}
\end{equation}
since $k_0,\;k_j\in\{1,\dots,4\}$ {\md and $j=1,\dots,n$.} Recalling that  $n>1\; (n\geq 2)$ {\md and $s \geq 2/\a +4$}, one gets the desired result
\begin{equation}
\label{A24}
\theta_0+\theta_1+\dots+\theta_n\geq\,\frac{n(s+r)-{\md (n+2)}}{s+r-1} \geq {\md n- \frac{2}{r+3+\frac2{\a}}} >  n-\alpha.
\end{equation}
\end{proof}

\subsection*{Acknowledgments} 
This work was done while H.M.
was visiting the Department of Mathematics of the University of California
at Santa Barbara whose hospitality he gratefully acknowledges.
H.M. was supported by the Overseas Research Fellowship Program by National Institute of Technology. F.L. was partially supported by CNPq and FAPERJ/Brazil.

\end{document}